\documentclass{article}
\usepackage[utf8]{inputenc}

\usepackage{hyperref}
\usepackage{amsthm}
\usepackage{amssymb}
\usepackage{enumerate}
\usepackage{tikz-cd}
\usepackage{mathtools}
\usepackage{cleveref}
\usepackage[top=3.75cm, bottom=3cm, left=3cm, right=3cm]{geometry}
\tikzcdset{arrow style=math font}
\usepackage{braket}
\usepackage{xcolor}
\usepackage{comment}
\usepackage{stmaryrd}
\usepackage[all,cmtip]{xy}

\def\C{\ensuremath{\mathbb{C}}}

\def\P{\ensuremath{\mathbb{P}}}

\def\Z{\ensuremath{\mathbb{Z}}}

\def\CC{\ensuremath{\mathcal C}}
\def\DD{\ensuremath{\mathcal D}}
\def\EE{\ensuremath{\mathcal E}}

\def\KK{\ensuremath{\mathcal K}}

\def\MM{\ensuremath{\mathcal M}}

\def\OO{\ensuremath{\mathcal O}}

\def\UU{\ensuremath{\mathcal U}}

\def\AAA{\mathfrak A}

\def\S{\ensuremath{\mathrm S}}

\def\Aut{\mathop{\mathrm{Aut}}\nolimits}
\def\Db{\mathop{\mathrm{D}^{\mathrm{b}}}\nolimits}
\def\Def{\mathop{\mathrm{Def}}\nolimits}
\def\Ext{\mathop{\mathrm{Ext}}\nolimits}
\def\Hom{\mathop{\mathrm{Hom}}\nolimits}
\def\Spec{\mathop{\mathrm{Spec}}\nolimits}
\def\ext{\mathop{\mathrm{ext}}\nolimits}
\def\Coh{\mathop{\mathrm{Coh}}\nolimits}
\def\Art{\mathop{\mathrm{Art}}\nolimits}
\def\RHom{\mathop{\mathrm{RHom}}\nolimits}
\def\Aut{\mathop{\mathrm{Aut}}\nolimits}
\def\Dp{\mathop{\mathrm{D}_{\text{pug}}}\nolimits}

\def\Forg{\mathop{\mathrm{Forg}}\nolimits}
\def\Ku{\mathop{\mathrm{Ku}}\nolimits}

\theoremstyle{plain}
  \newtheorem{thm}{Theorem}[section]
  \newtheorem{cor}[thm]{Corollary}
  \newtheorem{lem}[thm]{Lemma}
  \newtheorem{prop}[thm]{Proposition}

  \theoremstyle{definition}
\newtheorem{dfn}[thm]{Definition}
\newtheorem{rmk}[thm]{Remark}

\newtheorem{ex}[thm]{Example}

\newtheorem*{str}{Strategy of the proofs}
\newtheorem*{plan}{Plan of the paper}
\newtheorem*{ack}{Acknowledgements}

\title{Some remarks about deformation theory and formality conjecture}
\author{Huachen Chen, Laura Pertusi, Xiaolei Zhao}

\date{ }

\newcommand{\Addresses}{{
  \bigskip
  \footnotesize

  H.~Chen, \textsc{Department of Mathematics, University of California, Santa Barbara, South Hall, Santa Barbara, CA 93106, USA.}\par\nopagebreak
  \textit{E-mail address}: \texttt{huachen@ucsb.edu}.

  \medskip

  L.~Pertusi, \textsc{Dipartimento di Matematica F.\ Enriques, Universit\`a degli studi di Milano, Via Cesare Saldini 50, 20133 Milano, Italy.}\par\nopagebreak
  \textit{E-mail address}: \texttt{laura.pertusi@unimi.it} \par\nopagebreak
  \textit{URL}: \texttt{http://www.mat.unimi.it/users/pertusi/}.

  \medskip

  X.~Zhao, \textsc{Department of Mathematics, University of California, Santa Barbara, South Hall 6705, Santa Barbara, CA 93106, USA.}\par\nopagebreak
  \textit{E-mail address}: \texttt{xlzhao@ucsb.edu} \par\nopagebreak
  \textit{URL}: \texttt{https://sites.google.com/site/xiaoleizhaoswebsite/}.

}}

\begin{document}

\maketitle

\begin{abstract}
Using the algebraic criterion proved by Bandiera, Manetti and Meazzini, we show the formality conjecture for universally gluable objects with linearly reductive automorphism groups in the bounded derived category of a K3 surface. As an application, we prove the formality conjecture for polystable objects in the Kuznetsov components of Gushel--Mukai threefolds and quartic double solids. 
\end{abstract}

\section{Introduction}

The Formality Conjecture, as formulated by Kaledin and Lehn in \cite{KaledinLehn}, states that the differential graded algebra $\RHom(\mathcal{F},\mathcal{F})$ of the derived endomorphisms of a polystable sheaf $\mathcal{F}$ on a K3 surface (with respect to a generic polarization) is formal, i.e.\ it is quasi-isomorphic to its cohomology algebra. The conjecture has been proved in some cases by \cite{KaledinLehn} and \cite{Zhang}. Then it has been completely solved by Budur and Zhang in \cite{BudurZhang}, who showed this result more generally for every choice of the polarization and for polystable complexes in the bounded derived category with respect to a generic Bridgeland stability condition, and by Bandiera, Manetti and Meazzini in \cite{BMM2}, who showed the conjecture for polystable sheaves on a smooth minimal projective surface of Kodaira
dimension $0$. A further generalization of this conjecture has been recently proved by Arbarello and Saccà in \cite{AS_update}, for polystable objects in the bounded derived category of a K3 surface and in the Kuznetsov component of a cubic fourfold (without the genericity assumption on the stability condition). See also \cite{Davison} for the general setting of Calabi-Yau-2 categories.

The first result of this paper is an extension of Arbarello and Saccà's formality result in the case of K3 surfaces.

\begin{thm}[Theorem \ref{thm_formality}]
\label{thm_formality_intro}
 Let $X$ be a K3 surface and $F\in \Db(X)$ be an object such that $\Ext^i(F,F)=0$ for every $i<0$ and the automorphisms group $\emph{Aut}(F)$ is linearly reductive. Then the derived endomorphism Lie algebra $\RHom(F, F)$ is formal.
\end{thm}

As an application of Theorem \ref{thm_formality_intro}, we obtain a formality statement for polystable objects in certain Enriques categories that arise as semiorthogonal components in the bounded derived category of some Fano threefolds. More precisely, let $Y$ be a quartic double solid, which is a smooth double cover of $\P^3$ branched on a quartic surface. Its Kuznetsov component $\Ku(Y)$ is defined by the semiorthogonal decomposition
\begin{equation}
\Db(Y)=\langle \Ku(Y), \OO_Y, \OO_Y(1) \rangle    
\end{equation}
and is an Enriques category (in the sense of Definition \ref{definition-CY2-Enriques}). As a second example, consider a special Gushel--Mukai threefold $X$, which is a smooth double cover of a codimension three linear section of the Grassmannian $\text{Gr}(2,5)$ branched on quartic surface. Then there is a semiorthogonal decomposition
\begin{equation}
\Db(X)=\langle \Ku(X), \OO_X, \UU_X \rangle    
\end{equation}
where $\UU_X$ denotes the pullback to $X$ of the rank $2$ tautological bundle on $\text{Gr}(2,5)$ and $\Ku(X)$ is the Kuznetsov component, which is an Enriques category. In both situations, the branch locus is a K3 surface and its derived category is the CY2 cover of the given Enriques category (see Lemma \ref{lemma-Enriques-CY2}).

Our second result is the following general statement about Enriques categories. 

\begin{thm}[Theorem \ref{thm_formality_Enriques}]
\label{thm_formality_Enriques_intro}
Let $\CC$ be an Enriques category with CY2 cover category $\DD$. Let $\DD^{\emph{dg}}$ be a DG enhacement of $\DD$ with a $\Z/2\Z$-action and let $(\DD^{\emph{dg}})^{\Z/2\Z}$ be the associated enhancement of $\DD^{\Z/2\Z}$. Let $E$ be a an object in $\CC$ such that $\Hom_{\DD^{\emph{dg}}}(\Forg(E),\Forg(E))$ is formal, where $\Forg \colon \CC^{\emph{dg}} \cong (\DD^{\emph{dg}})^{\Z/2\Z} \to \DD^{\emph{dg}}$ is the forgetful functor. Then $\Hom_{\CC^{\emph{dg}}}(E, E)$ is formal.   
\end{thm}

Using Theorem \ref{thm_formality_Enriques_intro}, we show the formality conjecture for polystable objects in the Kuznetsov component of quartic double solids and Gushel--Mukai threefolds and fivefolds.

\begin{cor}[Corollary \ref{cor_qdsandspecGM3}] \label{cor_qdsandspecGM3_intro}
Let $X$ be a quartic double solid or a Gushel--Mukai threefold or fivefold. Let $\sigma$ be a Serre-invariant stability condition on the Kuznetsov component $\Ku(X)$ of $X$. Then $\RHom(E, E)$ is formal for every $\sigma$-polystable object $E \in \Ku(X)$.   
\end{cor}

From the point of view of geometry, the formality conjecture is useful in the study of the local structure of moduli spaces at singular points. In fact, it is well-known that the formality of the derived endomorphism algebra of an object $F$ implies the quadracity property (see also \cite{BMM1} for a result in the opposite direction). Thus Theorem \ref{thm_formality_Enriques_intro} provides a description of a formal neighbourhood of a singular point in the moduli space of semistable objects in an Enriques category, like the Kuznetsov component of a quartic double solid or of a special Gushel--Mukai threefold, as in Proposition \ref{prop_localstrmoduli_Enriques}.

\begin{str}
In \cite{BMM2} Bandiera, Manetti and Meazzini prove an algebraic criterion to ensure the formality of a differential graded Lie algebra, which involves the notion of quasi-cyclic DG-Lie algebra. Then they apply this criterion to the DG-Lie algebra of derived endomorphisms $\RHom(\mathcal{F},\mathcal{F})$ of a coherent sheaf $\mathcal{F}$ with linearly reductive automorphisms group on a smooth minimal surface of Kodaira dimension $0$. 

The proof of Theorem \ref{thm_formality_intro} is an application of Bandiera-Manetti-Meazzini's criterion. As in their case, a key point is the construction of a locally free replacement for $F$ with certain special properties which is done in Lemma \ref{lemma_Greplacement} and Lemma \ref{lemma_Greplacement2}. Theorem \ref{thm_formality_Enriques_intro} is a direct consequence of the formality transfer proved in \cite[Theorem 3.4]{Manetti2}. The application in Corollary \ref{cor_qdsandspecGM3_intro} makes use of the strongly uniqueness of DG enhancements known by \cite{LuntsOrlov, LPZ:enhancement} and the result in \cite{BudurZhang} about invariance of formality under this condition.
\end{str}

\begin{plan}
The paper is organized as follows. In Section 2 we review the deformation theory of an object $F$ in a full admissible subcategory of the bounded derived category of a smooth  projective variety using the language of DG-Lie algebras (see the setting we work on at the beginning of Section 2). In particular, we consider the Dolbeault DG-Lie algebra presentation of $\RHom(F,F)$, defined in \eqref{defofL} using a locally free replacement for $F$. We recall the definition of the Kuranishi map and we show how this is related to the local structure of a good moduli space. In Section 3, after recalling the algebraic approach of \cite{BMM2}, we prove Theorem \ref{thm_formality_intro} and Theorem \ref{thm_formality_Enriques_intro}, together with the application to quartic double solids and Gushel--Mukai varieties of odd dimension.  
\end{plan}

\begin{ack}
We are very grateful to Enrico Arbarello and Giulia Sacc\`a for their interest and for useful conversations. It is a pleasure to thank Francesco Meazzini for many interesting discussions and Paolo Stellari for helpful comments. We also thank Emanuele Macrì for his interesting questions and for motivating us to complete this work.

X.Z.\ is partially supported by the Simons Collaborative Grant 636187, NSF grant DMS-2101789, and NSF FRG grant DMS-2052665. This paper was written when the second and third author were attending the Junior Trimester program ``Algebraic geometry: derived categories, Hodge theory, and Chow groups'' at the Hausdorff Institute for Mathematics in Bonn. We would like to thank this institution for the warm hospitality.
\end{ack}




\section{Deformation theory and good moduli spaces} 

We work in the following setting:
\begin{itemize}
\item Let $X$ be a smooth projective variety defined over $\C$ and let $\Db(X):=\Db(\Coh(X))$ be the bounded derived category of coherent sheaves on $X$.
\item Let $\DD$ be a full admissible subcategory of $\Db(X)$, e.g.\ $\DD$ is a semiorthogonal component of $\Db(X)$.
\item Let $F \in \DD$ be a universally gluable object, i.e. $\Ext^i(F,F)=0$ for every $i <0$. For instance this holds when $F$ is an object in the heart of a bounded t-structure on $\DD$.
\end{itemize}
The aim of this section is to review the deformation theory of $F$ due to Lieblich \cite{Lieblich} and show it is equivalently described by the deformation theory of the DG-Lie algebra $\RHom(F,F)$ of derived endomorphisms of $F$ in the framework introduced by Manetti \cite{Manetti}. Note that this is well-known in the case of vector bundles \cite{FMM} and coherent sheaves \cite{FIM}. The generalization to universally gluable object $F \in \DD$ is a direct consequence of these results and the existence of a locally free replacement for $F$ (see Lemma \ref{lemma_Greplacement}). We write this explicitly for the sake of completeness. Finally we use deformation theory to study the local structure of a good moduli space parameterizing objects $F \in \DD$ as above, under the additional assumption that $F$ has linearly reductive automorphism group.

\subsection{Preliminaries on deformation theory and DG-Lie algebras} \label{sec_deftheory}
 
Let $\Art$ be the category of local Artinian $\C$-algebras. As in \cite[Definition 3.2.1]{Lieblich}, given $A \in \text{ob}(\Art)$ we say that a deformation of $F$ to $X_A:=X \times \Spec A$ is the data of a complex $F_A$ on $X_A$ together with an isomorphism $\varphi \colon F_A \otimes^{\mathbb{L}}_A \C \cong F$. Consider the functor
$$\Def_F \colon \Art \to \text{Set}$$
defined by associating to every object $A$ in $\Art$ the set $\Def_F(A)$ of equivalence classes of deformations of $F$ to $X_A$, with respect to the following equivalence relation: two deformations $(F_A, \varphi)$, $(F'_A, \varphi')$ are equivalent if there exists an isomorphism $\psi \colon F_A \cong F_A'$ such that $\varphi' \circ \psi= \varphi$. Note that $F_A$ is in $\Db(X_A)$ by \cite[Lemma 3.2.4]{Lieblich}. By \cite[Theorem 3.1.1 and Proposition 3.5.1]{Lieblich} the functor $\Def_F$ is a deformation functor, whose tangent space $\Def_F(\C[t]/(t^2))$ is $\Ext^1(F,F)$, and the obstruction to the deformation is in $\Ext^2(F,F)$.

We remark 
that a deformation $F_A$ of $F$ to $X_A$ belongs to the base change category $\DD_A:=\DD \otimes_{\Db(\Spec \C)} \Db(\Spec A)$ if and only if $F \in \DD$ by \cite[Lemma 9.3]{BLM+} (see also \cite[Corollary 5.9]{Kuz_basechange}). 

On the other hand, it is possible to use the powerful theory of \emph{differential graded Lie (DG-Lie) algebras} to study the deformation theory of $F$. We refer to \cite{Manetti} for a survey on DG-Lie algebras and their associated deformation functors, recalling only the main definitions we need in the next (we work over $\C$, but the definitions are stated over a field of characteristic $0$).

\begin{dfn}
A DG-Lie algebra is a graded vector space $L:=\oplus_i L^i$ equipped with a differential $d: L^i \to L^{i+1}$ and a Lie bracket $[-,-]: L^i\otimes L^j \to L^{i+j}$ satisfying
\begin{itemize}
    \item $[f,g] = (-1)^{|f||g|+1}[g,f]; $
    \item $d[f,g] = [df, g] + (-1)^{|f|}[f,dg];$
    \item $(-1)^{|f||h|}[f,[g,h]]+(-1)^{|g||f|}[g,[h,f]]+(-1)^{|h||g|}[h,[f,g]]=0,$
\end{itemize}
for any homogeneous elements $f,g,h \in L$ of degree $|f|,|g|, \text{and}\ |h|,$ respectively. We shall denote by $H^i(L)$ the $i$-th cohomology of $L$ with respect to the differential $d.$
\end{dfn}
Associated to a DG-Lie algebra $L$ we have the following functors (see \cite[Section 3]{Manetti}). The \emph{Maurer-Cartan functor} $\text{MC}_L \colon \Art \to \text{Set}$ is defined by
$$\text{MC}_L(A)= \lbrace x \in L^1 \otimes \mathfrak{m}_A \colon dx + \frac{1}{2}[x,x]=0 \rbrace,$$
where $\mathfrak{m}_A$ is the maximal ideal of $A$. Equivalently, $\text{MC}_L(A)$ is the set of elements in $L^1 \otimes \mathfrak{m}_A$ satisfying the Maurer-Cartan equations. The \emph{Gauge functor} $\text{G}_L \colon \Art \to \text{Grp}$ is defined by 
$$\text{G}_L(A)=\exp(L^0 \otimes \mathfrak{m}_A),$$
where $\exp$ is the exponential functor applied to the nilpotent Lie algebra $L^0 \otimes \mathfrak{m}_A$ with values in the category of groups (see \cite[Example 2.4.2]{Manetti}). The \emph{deformation functor} associated to $L$ is 
$$\Def_L:= \text{MC}_L / \text{G}_L,$$
where the quotient is taken with respect to the action of the group functor $\text{G}_L$ on $\text{MC}_L$. As computed in \cite[Section 3]{Manetti}, the tangent space to $\Def_L$ is $H^1(L)$ and the obstructions are in $H^2(L)$.

\subsection{DG-Lie algebra of derived endomorphisms}
Coming back to our object $F \in \DD$, we need to fix a finite locally free replacement $\mathcal E^\cdot$ of $F$. The existence of such a replacement is probably known to the experts, but in order to study the formality property in the next sections we need one with special properties, as follows.
\begin{lem}\label{lemma_Greplacement}
Let $X$ be a smooth projective variety and $F \in \Db(X)$ be an object in the bounded derived category of $X,$ let $\emph{Aut}(F)$ be the group of automorphisms of $F.$ There exists an $\emph{Aut}(F)$-equivariant locally free replacement $\mathcal E^\cdot$ of $F.$ 
\end{lem}

\begin{proof}
Denote by $G:= \Aut(F)$. Choose a bounded complex representative for $F$, and without loss of generality, we assume that $\mathcal F^i = 0$ for $i\notin [0,n]$. We divide the proof into several steps.\\
\\
\textbf{First Step:} Consider $$\mathcal F^{n-1} \twoheadrightarrow \mathcal I^{n-1} \hookrightarrow \mathcal F^n \twoheadrightarrow \mathcal I^n\to 0,$$ where $\mathcal I^{n-1}$ is the image of $\mathcal F^{n-1}\to \mathcal F^n$ and $\mathcal I^n \cong H^n (F)$ is the $n$-th cohomology sheaf of $F$. Let $\mathcal O_X(1)$ be a polarization on $X$. For a sufficiently large integer $k_n$ we define $\mathcal E^n:= H^0(X,\mathcal I^n(k_n))\otimes \mathcal O_X(-k_n)$ and thus have an exact sequence $$0\to \mathcal K^n \to \mathcal E^n\xrightarrow{ev_n} \mathcal I^n \to 0,$$ where $ev_n$ is the evaluation map and $\mathcal K^n:= ker(ev_n)$ is its kernel. The group of automorphisms $G$ acts on $\mathcal I^n$ and $\mathcal E^n,$ rendering the natural map $ev_n$ $G$-equivariant. Thus, $G$ acts on the kernel $\mathcal K^n$ as well. If $k_n$ is sufficiently large such that $\Ext^1(\mathcal O_X(-k_n), \mathcal I^{n-1})=0$, then the map $ev_n$ admits a lifting $\Tilde{ev}_n$ to $\mathcal F^n$ 
\begin{equation*}
    \begin{tikzcd}
    &  \mathcal K^n \arrow[rr, hook] \arrow[d, dashed] & & \mathcal E^n  \arrow[d, "ev_n"] \arrow[dl, dashed, "\Tilde{ev}_n"] & \\
    \mathcal F^{n-1} \arrow[r, two heads] & \mathcal{I}^{n-1} \arrow[r, hook] & \mathcal F^n \arrow[r]&  \mathcal I^n \arrow[r]& 0,
    \end{tikzcd}
\end{equation*} 
and induces a map $\mathcal K^n \to \mathcal I^{n-1}$.\\
\\
\textbf{Next Steps:} Now we can continue the construction of $\mathcal E$ by iterating the construction in the first step. Suppose that we have $\mathcal E^{\geq i}=[0\to \mathcal E^i \to \ldots \to \mathcal E^n \to 0]$ as a $G$-equivariant complex of locally free sheaves concentrated in degrees $[i, n]$, and it admits a chain map $\mathcal E^{\geq i} \to F$, such that the induced maps $H^j(\mathcal E^{\geq i}) \xrightarrow{\sim} H^j(F)$ are isomorphic for all $j>i$ and $H^{i}(\mathcal E^{\geq i}) \twoheadrightarrow H^i(F)$ is surjective. (Note that $\mathcal E^{\geq n}:= \mathcal E^n [-n]$ constructed above satisfies this condition, so in the first step we have achieved for $i=n$.)

Then let $\mathcal C^{<i}:= \text{cone}(\mathcal E^{\geq i}\to F)$ be the cone.  The cohomology sequence of this distinguished triangle gives the following $G$-equivariant exact sequences $$0 \to H^{i-1}(F) \to H^{i-1}(\mathcal C^{<i}) \to H^{i}(\mathcal E^{\geq i}) \twoheadrightarrow H^i(F)  \to H^i(\mathcal C^{<i}) = 0$$ and 
\begin{equation}\label{link}
    0 \to H^{i-1}(F) \to H^{i-1}(\mathcal C^{<i}) \twoheadrightarrow \mathcal K^i \to 0,
\end{equation}
where $\mathcal K^i:= ker(H^{i}(\mathcal E^{\geq i}) \twoheadrightarrow H^i(F))$ is the kernel. Note that we have a diagram 
\begin{equation*}
    \begin{tikzcd}
     & & \mathcal K^{i} \arrow[r, hook] \arrow[d, dashed] & H^i(\mathcal E^{\geq i}) \arrow[r, hook] \arrow[d] & \mathcal E^i  \arrow[d, ] \arrow[r] & \ldots\\
  \ldots  \arrow[r] & \mathcal F^{i-1} \arrow[r, two heads] & \mathcal I^{i-1} \arrow[r, hook] & \mathcal{Z}^{i} \arrow[r, hook] & \mathcal F^i \arrow[r]&   \ldots,
    \end{tikzcd}
\end{equation*}
where $\mathcal Z^i := ker(d^i:\mathcal F^i \to \mathcal F^{i+1})$ and $\mathcal I^{i-1}:= Im(d^{i-1}: \mathcal F^{i-1}\to \mathcal F^i)$ are the kernel and image. The map $\mathcal K^i \to \mathcal I^{i-1}$ is induced by the factorization $H^i(\mathcal E^{\geq i}) \to \mathcal Z^i \to H^i(F).$

Define $\mathcal E^{i-1}:= H^0(X, H^{i-1}(\mathcal C^{<i})(k_{i-1}))\otimes \mathcal O_X(-k_{i-1})$ for some sufficiently large number $k_{i-1}$ and let $ev_{i-1}: \mathcal E^{i-1} \twoheadrightarrow H^{i-1}(\mathcal C^{<i}) \twoheadrightarrow \mathcal K^i \to \mathcal I^{i-1}$ be the composition. Then, we have a lifting $\Tilde{ev}_{i-1}$
\begin{equation*}
    \begin{tikzcd}
     & \mathcal E^{i-1} \arrow[rd, , "ev_{i-1}"] \arrow[r, two heads] \arrow[d, dashed, "\Tilde{ev}_{i-1}"] & \mathcal K^i \arrow[r, hook] \arrow[d] & \mathcal E^i \arrow[r]  \arrow[d] & \ldots \\
  \ldots \arrow[r] & \mathcal F^{i-1} \arrow[r, two heads] & \mathcal{I}^{i-1} \arrow[r, hook] & \mathcal F^i \arrow[r]&  \ldots.
    \end{tikzcd}
\end{equation*} 
Thus, we obtain a $G$-equivariant complex $ \mathcal E^{\geq i-1}:= [0\to \mathcal E^{i-1}\to \ldots \to \mathcal E^n \to 0]$, which admits a chain map to $\mathcal F$. Note that $H^i(\mathcal E^{\geq i-1})\cong coker(\mathcal K^i \to H^i(\mathcal E^{\geq i})) \cong H^i(F)$ and $H^{i-1}(\mathcal E^{\geq i-1})\twoheadrightarrow H^{i-1}(F)$ is surjective because of the surjectivity of $\mathcal E^{i-1} \twoheadrightarrow H^{i-1}(\mathcal C^{<i})$ and \Cref{link}.\\
\\
\textbf{Final step:} Once we achieve the previous steps for $i=0$, we have a locally free complex $\mathcal E^{\geq 0}$ with equivariant $G$-actions on each $\mathcal E^i$, with $H^j(\mathcal E^{\geq 0}) \xrightarrow{\sim} H^j(F)$ are isomorphic for all $j>0$ and $H^{0}(\mathcal E^{\geq 0}) \twoheadrightarrow H^0(F)$ is surjective. Since $X$ is smooth and projective, we can continue resolving the kernel of $H^{0}(\mathcal E^{\geq 0}) \twoheadrightarrow H^0(F)$ by bounded complex of equivariant locally free sheaves. This gives the desired replacement.
\end{proof}

As in \cite[Section 5]{BMM2}, we consider the Dolbeault DG-Lie algebra presentation of $\RHom(F, F)$: given a finite locally free replacement $\mathcal E^\cdot$ of $F$, for instance the one in Lemma \ref{lemma_Greplacement}, define
\begin{equation}\label{defofL}
    L:=  \bigoplus_{q, r,s } A^{0,q}(\mathcal Hom(\mathcal E^r,\mathcal E^s)),
\end{equation}
where $A^{0,q}(\mathcal Hom(\mathcal E^r,\mathcal E^s))$ is the space of global $(0,q)$-forms taking values in $\mathcal Hom(\mathcal E^r,\mathcal E^s).$ Note that $L$ inherits a standard differential 
\begin{equation}\label{defofd}
d(\omega^{0,q}):= (-1)^q\bar{\partial}(\omega^{0,q})+d_{\mathcal E}\circ(\omega^{0,q}),
\end{equation}
where $\omega^{0,q}\in A^{0,q}(Tot^\cdot(\mathcal Hom(\mathcal E^\cdot, \mathcal E^\cdot)))$ is a $(0,q)$-form and $d_{\mathcal E}$ is the differential of the total complex $Tot^\cdot(\mathcal Hom(\mathcal E^\cdot, \mathcal E^\cdot))$. The wedge product $\wedge$ of $(0,q)$-forms naturally gives $L$ a DG-algebra structure with respect to the differential $d$, and moreover, induces a Lie bracket 
\begin{equation}\label{defofbracket}
[f,g]:=f\wedge g - (-1)^{|f||g|}g\wedge f. 
\end{equation}
The triple $(L, d, [-,-])$ is then a DG-Lie algebra representing $\RHom(F,F)$. 

We remark the following property which will be used in proof of Theorem \ref{thm_formality}.
\begin{lem}\label{lemma_Greplacement2}
Consider the $\emph{Aut}(F)$-equivariant locally free replacement $\EE^\cdot$ of $F \in \Db(X)$ constructed in Lemma \ref{lemma_Greplacement}. Then the DG-Lie algebra $L$ defined in \eqref{defofL} using such a replacement $\mathcal E^\cdot$ admits an $\emph{Aut}(F)$-action, such that each degree of $L$ is a rational representation of $\emph{Aut}(F).$ 
\end{lem}
\begin{proof}
Set $G :=\Aut(F)$. By \cite[Lemma 3.5]{BMM1}, if $U\subset X$ is any open subset, then $\mathcal E^i(U)$ is a rational representation of $G$ with finite support for every subgroup of $G.$ Then, by \cite[Theorem 3.8]{BMM1}, the DG Lie algebra $L:=  \bigoplus_{q, r,s } A^{0,q}(\mathcal Hom(\mathcal E^r,\mathcal E^s))$ has the desired property.
\end{proof}

In the next lemma we show that the deformation theory of $F$ is controlled by the deformation theory of the DG-Lie algebra $\RHom(F,F)$. This is an immediate consequence of the results in \cite{FIM} for coherent sheaves.

\begin{lem}
\label{lemma_deformationfunctors}
There is an isomorphism between the deformation functors
$$\Def_F \cong \Def_L.$$
\end{lem}
\begin{proof}
The result holds for vector bundles and coherent sheaves by \cite{FMM}, \cite{FIM}. Since $\EE^\cdot \cong F$ in $\Db(X)$ we have $\Def_F \cong \Def_{\EE^\cdot}$. By \cite[Section 2]{FIM} we have the isomorphism $\Def_{\EE^\cdot} \cong H^1_{\text{Ho}}(X;\exp \mathcal Hom(\EE^\cdot,\EE^\cdot))$, where the latter is a \v{C}ech type functor associated to the sheaf of DG-Lie algebras $\mathcal Hom(\EE^\cdot,\EE^\cdot)$ (see \cite[Section 3.1]{FIM} for the precise definition). We argue as explained in \cite[Section 5]{FIM}. Since $\Ext^{i}(F,F)=0$ for every $i<0$, by \cite[Theorem 4.11]{FIM} we have $H^1_{\text{Ho}}(X;\exp \mathcal Hom(\EE^\cdot,\EE^\cdot)) \cong \Def_{[\mathcal Hom(\EE^\cdot, \EE^{\cdot})]}$ and by \cite[Theorem 3.7]{FIM} we have $\Def_{[\mathcal Hom(\EE^\cdot, \EE^{\cdot})]} \cong \Def_{A^{0,*}(\mathcal Hom(\EE^\cdot, \EE^\cdot))}$. This implies the statement. 
\end{proof}

\begin{rmk}
Consider the locally free replacement $\EE^\cdot$ of Lemma \ref{lemma_Greplacement} and \ref{lemma_Greplacement2}. Since the resolution $\EE^\cdot$ is $\Aut(F)$-equivariant, one can check that the isomorphism in Lemma \ref{lemma_deformationfunctors} is $\Aut(F)$-equivariant too.   
\end{rmk}

\subsection{Kuranishi map} \label{sec_kuranishi}
We keep the setting introduced at the beginning of this section and we consider $F \in \DD$ whose group of automorphisms $G:=\Aut(F)$ is linearly reductive. 

Using the obstruction to the deformation of $F$ it is possible to define the well-known Kuranishi map. An explicit construction of a $G$-equivariant Kuranishi map is described in \cite[Appendix]{LehnSorger} if $F$ is a polystable sheaf on a K3 surface. By Lemma \ref{lemma_deformationfunctors}, the Kuranishi map of $F$ is identified with the Kuranishi map of the DG-Lie algebra $L$ defined in \eqref{defofL}, so we recall its definition using this formalism as in \cite[Section 4]{Manetti}.

We use the following notation. If $V$ is a $\C$-vector space we denote by $\widehat{V} \colon \text{Art} \to \text{Set}$ the functor defined on objects by $\widehat{V}(A)=V \otimes \mathfrak{m}_A$ and analogously on morphisms. 

Given a DG-Lie algebra $(L, d, [-,-])$ we set $Z^n(L)=\lbrace x \in L^n \colon dx=0 \rbrace$ and $B^n(L)=\lbrace dx \colon x \in L^{n-1} \rbrace$ for $n \in \Z$. We can choose a splitting
\begin{equation}
\label{eq_splitting} 
Z^n(L)= B^n(L) \oplus H^n(L) \quad \text{and} \quad L^n= Z^n(L) \oplus K^n.
\end{equation}
Note that $d \colon K^n \to B^{n+1}(L)$ is an isomorphism. Then we denote by $\delta \colon L^{n+1} \to L^n$ the linear composition of the projection $L^{n+1} \to B^{n+1}(L)$ with kernel $H^n(L) \oplus K^n$, the inverse $d^{-1} \colon B^{n+1}(L) \xrightarrow{\cong} K^n$, and the inclusion $K^n \hookrightarrow L^n$, i.e.\ $\delta$ sits in the commutative diagram
$$
\xymatrix{
L^{n+1} \ar[d] \ar[r]^\delta &L^n\\
B^{n+1}(L) \ar[r]^-{d^{-1}}          &K^n \ar@{^{(}->}[u]}.
$$
Let $H \colon L^n \to H^n(L)$ be the projection with kernel $B^n(L) \oplus K^n$. By \cite[Lemma 4.2]{Manetti} we have the isomorphism of functors $\phi \colon \widehat{L^1} \to \widehat{L^1}$ given by
$$x \in L^1 \otimes \mathfrak{m}_A \mapsto \phi(x)=x + \frac{1}{2}\delta[x,x] \in L^1 \otimes \mathfrak{m}_A.$$

The \emph{Kuranishi map} is the natural transformation
$$\kappa \colon \widehat{H^1(L)} \to \widehat{H^2(L)}$$
defined by
$$x \in H^1(L) \otimes \mathfrak{m}_A \mapsto \kappa(x)=H([\phi^{-1}(x), \phi^{-1}(x)]) \in H^2(L) \otimes \mathfrak{m}_A.$$ 
The \emph{Kuranishi functor} $\kappa^{-1}(0) \colon \Art \to \text{Set}$ is defined by
$$\kappa^{-1}(0)(A)= \lbrace x \in H^1(L) \otimes \mathfrak{m}_A \colon \kappa(x)=H([\phi^{-1}(x), \phi^{-1}(x)])=0 \rbrace.$$
By \cite[Theorem 4.7]{Manetti} there exists an \'etale morphism $\nu \colon \kappa^{-1}(0) \to \Def_L$.

Now consider the DG-Lie algebra $L$ in \eqref{defofL}. Then $H^1(L) \cong \Ext^1(F,F)$ and $H^2(L) \cong \Ext^2(F,F)$. As explained in \cite[Theorem 5.1]{BMM2}, we can choose a splitting of the form \eqref{eq_splitting} where $H^n(L)$ and $K^n$ are $G$-invariant. Indeed, by Lemma \ref{lemma_Greplacement2} and the assumption that $G$ is linearly reductive, we can extend $H^0(L) \cong \Hom(F,F) \subset Z^0(L) \subset L^0$ to a $G$-equivariant splitting of $L$ with the desired properties. As a consequence $H$ and $\delta$ are $G$-equivariant maps and then the Kuranishi map is $G$-equivariant. In particular, there is an action of $G$ on the Kuranishi functor and the \'etale morphism $\nu$ is $G$-equivariant. 

Since $H^1(L)$ is finite-dimensional, the Kuranishi functor is prorepresentable. More precisely, we follow the computation in \cite[Section 15.6]{Manetti_book}. Consider the polynomial ring $R:=\C[H^1(L)]$ over $H^1(L)$ and let $\widehat{R}$ be the completion of the ring $R$ with respect to the maximal ideal $\mathfrak{m}$ of polynomial functions vanishing at $0$. Consider the functor $h_{\widehat{R}}:= \Hom(\widehat{R}, -) \colon \Art \to \text{Set}$. Then we have the isomorphism of functors $h_{\widehat{R}} \cong \widehat{H^1(L)}$ which, in coordinates fixing a basis $e_1, \dots, e_m$ for $H^1(L)$ and writing $\widehat{R} \cong \C[[r_1, \dots, r_m]]$, is defined by
$$(f \colon \widehat{R} \to A) \quad  \mapsto \quad \sum_{i=1}^m e_i \otimes f(r_i) \in H^1(L) \otimes \mathfrak{m}_A.$$
Then we see that $f \colon \widehat{R} \to A$ corresponds to an element in $\kappa^{-1}(0)(A)$ if and only if 
$$\kappa(f) \colon H^1(L) \otimes \mathfrak{m} \to H^2(L) \otimes \mathfrak{m}_A$$ is the zero map. Write 
$$\kappa(\sum_{i=1}^m e_i \otimes r_i)= \sum_j v_j \otimes g_j(r_1, \dots,r_m) \in H^2(L) \otimes \mathfrak{m}_A$$
and denote by $\mathfrak{a}$ the ideal in $\widehat{R}$ generated by the power series $g_j$'s. Then the above computation shows that $$\kappa^{-1}(0) \cong h_{\widehat{R}/ \mathfrak{a}},$$
namely the Kuranishi functor is prorepresented by $\widehat{R}/\mathfrak{a}$ and
$$(\kappa^{-1}(0),0)=\text{Spf}(\widehat{R}/\mathfrak{a})=\text{colim}_{n}\Spec((\widehat{R}/ \mathfrak{a})/\mathfrak{m}^n)$$
as formal schemes (here $\mathfrak{m}$ is the maximal ideal of $\widehat{R}/\mathfrak{a}$). Using $\nu$ we have that there exists a $G$-equivariant formal deformation $(\widehat{F},\widehat{\varphi})$ of $F$ having the versality property, parametrized by $(\kappa^{-1}(0),0)$.

We denote by $\kappa_2$ the quadratic part of the Kuranishi map, which is given by the Yoneda product $\kappa_2(e)=e \mathbin{\smile} e$ for $e \in \Ext^1(F,F)$. We say that $F$ satisfies the \emph{quadracity property} if 
$$\kappa^{-1}(0) \cong \kappa^{-1}_2(0),$$
namely that the base space of the Kuranishi family is cutted out by quadratic equations. 

\subsection{Local structure of a good moduli space}
We keep the notation and assumptions introduced at the beginning of this section. Our goal is to describe the local structure of a good moduli space parametrizing universally gluable objects in $\DD$ using the deformation theory and the Kuranishi map.

Recall that given a flat, proper, finitely presented morphism of schemes $Y \to T$, a $T$-perfect complex $F \in \mathrm{D}(Y)$ is \emph{universally gluable} if for every point $t \in T$ we have $\Ext^i(F_t,F_t)=0$ for $i <0$ (see \cite[Definition 2.1.8,  Proposition 2.1.9]{Lieblich}). We denote by $\Dp(Y) \subset \mathrm{D}(Y)$ the full subcategory of universally gluable $T$-perfect
objects.

In our setting, consider the functor
$$\MM_{\text{pug}}(\DD) \colon (\text{Sch})^{\text{op}} \to \text{Gpd}$$
from the opposite category of schemes over $\C$ to the category of groupoids, which associates to $T$ in $\text{ob}(\text{Sch})$ the groupoid $\MM_{\text{pug}}(\DD)(T)$ of all $F \in \Dp(X_T)$ such  that $F_t \in \DD_t$ for every $t \in T$. By \cite[Proposition 9.2]{BLM+} $\MM_{\text{pug}}(\DD)(T)$ is an algebraic stack locally of finite presentation and the canonical morphism $\MM_{\text{pug}}(\DD) \to \MM_{\text{pug}}(\Db(X))$ is an open immersion.

Let $\MM$ be an open substack of $\MM_{\text{pug}}(\DD)$ and assume it admits a good moduli space 
$$\pi \colon \MM \to M$$ 
in the sense of \cite{Alper}. A closed point $x \in \MM$ represents the orbit of a universally gluable objects $F \in \DD$ with respect to the action of the stabilizer group $G_x \cong \Aut(F)$.

\begin{ex} \label{ex_modulistability}
Assume that there exists a proper stability condition $\sigma$ on $\DD$ with respect to the numerical Grothendieck group $K_{\text{num}}(\DD)$, i.e.\ $\sigma$ is a full numerical stability condition on $\DD$. For $v \in K_{\text{num}}(\DD)$, we denote by $\MM_\sigma(v)$ the moduli stack parametrizing $\sigma$-semistable objects in $\DD$. By \cite[Theorem 21.24]{BLM+}, \cite{AHLH} $\MM_{\sigma}(v)$ admits a good moduli space 
$$\pi \colon \MM_{\sigma}(v) \to M_\sigma(v),$$ 
which is a proper algebraic space. 

A closed point $x \in \MM_{\sigma}(v)$ determines the orbit of a $\sigma$-semistable object $F$ with numerical class $v$ with respect to the action of the stabilizer $G_x \cong \Aut(F)$. The point $\pi(x) \in M_\sigma(v)$ determines the S-equivalence class of $F$. 
\end{ex}

Let $T_{\MM,x}$ be the tangent space to $\MM$ at $x$. Note that by definition  
$$T_{\MM,x} \cong T_{\MM_{\text{pug}}(\DD)} \cong \Def_F(\C[t]/(t^2)) \cong \Ext^1(F,F)$$
where $\Def_F$ is the deformation functor of $F$ in $\DD$ introduced in Section \ref{sec_deftheory}. Recall from Section \ref{sec_kuranishi} the notation $R:=\C[\Ext^1(F,F)]$, the completion $\widehat{R}$ of the ring $R$ at the maximal ideal $\mathfrak{m}$ of polynomial functions vanishing at $0$ and the ideal $\mathfrak{a} \subset \widehat{R}$ such that $\widehat{R}/\mathfrak{a}$ prorepresents the Kuranishi map.

The following property is well-known in the case of moduli spaces of semistable sheaves on a K3 surface, and generalizes \cite[Lemma 3.2]{LPZ:elliptic} to this more general setting. 

\begin{lem}
\label{lemma_localstructuremoduli}
With the above notation, assume that the stabilizer group $G_x$ of $x$ is linearly reductive. Then 
$$\widehat{\OO}_{M,\pi(x)} \cong (\widehat{R}/\mathfrak{a})^{G_x} \cong \widehat{R}^{G_x} /(\mathfrak{a} \cap \widehat{R}^{G_x}).$$
\end{lem}
\begin{proof}
Since $G_x$ is linearly reductive, by \cite[Lemma 3.2]{LPZ:elliptic}, $G_x$ acts on the quotient $\widehat{R}/ \mathfrak{a}$ and $(\widehat{R}/ \mathfrak{a})^{G_x} \cong \widehat{R}^{G_x}/ (\mathfrak{a} \cap \widehat{R}^{G_x})$.  

Note that $(\widehat{R}/ \mathfrak{a})^{G_x}$ is a complete local ring. Indeed, $\widehat{R}^{G_x} \cong \widehat{R^{G_x}}$ is a complete Noetherian local ring. It follows that the quotient $\widehat{R}^{G_x}/ (\mathfrak{a} \cap \widehat{R}^{G_x})$ is also a complete local ring, as we claimed.

Consider the stack $\KK:=[\Spec{(\widehat{R}/\mathfrak{a})}/G_x]$, whose good moduli space is $\KK \to \Spec(\widehat{R}/\mathfrak{a})^{G_x}$, and denote by $\KK_n$ the $n$-th thickening of $\KK$ at $0$. As a consequence of the above computation, by \cite[Theorem 1.3]{AlperHallRydh} the stack $\KK$ is coherently complete along $x$. This implies by  \cite[Corollary 2.6]{AlperHallRydh} that the compatible collection of morphisms $\KK_n \to \MM$, or equivalently the collection of equivariant compatible objects $(F_n, \varphi_n) \in \Def_F((\widehat{R}/\mathfrak{a})/\mathfrak{m}^{n+1})$ for every $n$, effectivizes to $\KK \to \MM$. Then the same argument as in \cite[Lemma 3.2]{LPZ:elliptic} implies the desired statement.
\end{proof}

\section{Formality results and moduli spaces}

In this section we prove Theorems \ref{thm_formality_intro} and \ref{thm_formality_Enriques_intro}. Then we explain the applications to the study of moduli spaces and to the formality conjecture in the case of Kuznetsov components of Gushel--Mukai threefolds and quartic double solids.

\subsection{Preliminaries on formality and quasi-cyclic DG-Lie algebras}

In this section, we review an algebraic approach introduced by Bandiera-Manetti-Meazzini \cite{BMM2} to the formality of DG-Lie algebras.

\begin{dfn}
A DG-Lie algebra $(L, d, [-,-])$ is \textit{formal} if it is quasi-isomorphic to its cohomology DG-Lie algebra $( H^\cdot(L), 0, [-,-]).$
\end{dfn}

A central notion in Bandiera-Manetti-Meazzini's approach is that of a quasi-cyclic DG-Lie algebra.

\begin{dfn}\label{quasicyclicdgla}
A DG-Lie algebra $(L, d, [-,-])$ with finite dimensional cohomology is called quasi-cyclic of degree $n$ if it admits a symmetric bi-linear pairing $(-,-): L\otimes L \to \mathbb C[-n]$ of degree $-n$ such that for any elements $f, g, h \in L,$ we have
\begin{enumerate}[(i)]
    \item $(df, g) + (-1)^{|f|}(f, dg)=0,$ if $f$ is homogeneous of degree $|f|$;
    \item $([f,g], h)=(f, [g,h]);$
    \item the induced pairing on cohomology $(-,-): H^\cdot(L)\otimes  H^\cdot(L) \to \mathbb C[-n]$ is non-degenerate.
\end{enumerate}

\end{dfn}

In the next section we will apply the following criterion for formality.

\begin{thm}[{\cite[Theorem 1.2]{BMM2}}]\label{BMMthm}
Let $(L, d, [-,-], (-,-))$ be a quasi-cyclic DG-Lie algebra of degree $n\leq 2$. Assume it admits a splitting $L= K \oplus H \oplus d(K)$ such that 
\begin{enumerate}[{(a)}]
    \item $H^i=0$ for $i<0$;
    \item $H^0 \subset L^0$ is closed with respect to the bracket $[-,-]$;
    \item for all $i>0,$ $H^i, K^i \subset L^i$ are $H^0$-submodules with respect to the adjoint action.
\end{enumerate}
Then the DG-Lie algebra $(L,d,[-,-])$ is formal.
\end{thm}

\subsection{K3 surfaces}

Let $X$ be a K3 surface and fix an object $F \in \Db(X)$. Recall the DG-Lie algebra $(L, d, [-,-])$ defined in \eqref{defofL} which is representing $\RHom(F,F)$, for $F \in \Db(X)$. Then a nondegenerate degree $2$ bilinear form on $L$ can be defined as done in \cite[Section 5]{BMM2} for coherent sheaves. More precisely, given $f, g\in L,$ define the pairing of $f$ and $g$ as
\begin{equation}\label{defofpairing_surface}
    (f,g):= \int_Y Tr(f\wedge g) \wedge \omega,
\end{equation}
where $Tr(-): L \to \oplus_q A^{0,q}(Y)$ is the trace map and $\omega \in H^{0}(Y,\Omega_Y^2)\cong \mathbb C$ is the unique (up to scaling) nontrivial class. Then $(L, d, [-,-], (-,-))$ is a degree $2$ quasi-cyclic DG-Lie algebra representative of $\RHom(F, F).$

\begin{thm}
\label{thm_formality}
 Let $X$ be a K3 surface and $F\in \Db(X)$ be a universally gluable object and the automorphisms group $\emph{Aut}(F)$ is linearly reductive. Then the derived endomorphism Lie algebra $\RHom(F, F)$ is formal.
\end{thm}
\begin{proof}
By \eqref{defofpairing_surface}, \Cref{lemma_Greplacement} and \Cref{lemma_Greplacement2}, we have a degree $2$ $\Aut(F)$-equivariant quasi-cyclic DG-Lie algebra $(L, d, [-,-], (-,-))$ representing $\RHom(F, F),$ where each degree of $L$ is a rational representation of $\Aut(F)$.  

We then argue as in \cite[Theorem 5.1]{BMM2}. Note that $G:=\Aut(F)$ acts faithfully on the group of $0$-cycles $Z^0(L) \subset L^0.$ It induces a $G$-equivariant Lie algebra embedding $\iota \colon \Hom(F, F) \hookrightarrow Z^0(L).$ Composing $\iota$ with the quotient $Z^0(L)\to H^0(L)\cong \Hom(F, F)$ yields an isomorphism, and therefore a $G$-invariant splitting $Z^0(L) = H^0(L) \oplus d(L^{-1})$, since $Z^0(L)\subset L^0$ is a rational representation of the linearly reductive group $G.$ Similarly, for each $i$, we have $L^i = Z^i(L)\oplus K^i$ for some $G$-invariant subspace $K^i\subset L^i$ and then $Z^i(L)= H^i(L) \oplus d(K^{i-1}).$ This gives a decomposition of $L$ satisfying the conditions in \Cref{BMMthm}, and therefore $\RHom(F,F)$ is formal.  
\end{proof}

\begin{rmk}
Theorem \ref{thm_formality} has been proved by \cite{BMM2} for coherent sheaves with linearly reductive automorphisms group, by \cite{BudurZhang} for polystable sheaves with respect to an arbitrary polarization and for complexes in the bounded derived category which are semistable with respect to a generic Bridgeland stability condition, and by \cite[Theorem 3.2]{AS_update} for polystable complexes with respect to any stability condition. See also \cite{Davison}.
\end{rmk}

Now consider a Mukai vector $v$ in the algebraic Mukai lattice $K_{\text{num}}(\Db(X))$ and a full numerical stability condition $\sigma$ on $\Db(X)$. Using Theorem \ref{thm_formality} one obtains a description of the local structure of the moduli space $M_\sigma(v)$ at the point corresponding to $F$. 

Let us first recall the construction of the quiver associated to a polystable object $F$ in $M_\sigma(v)$, which has been introduced in \cite[Proposition 6.1]{AS}, \cite[Section 4.3]{AS_update}. Write $F=\oplus_{i=1}^s F_i \otimes V_i$, where $V_1, \dots, V_s$ are vector spaces of dimension $n_1, \dots, n_s$ and $F_1, \dots, F_s$ are distinct $\sigma$-stable objects. Set $G :=\Aut(F)$ and denote by $\mathfrak{g}$ its Lie algebra. The quiver $Q(F)=(I,E)$ associated to $F$ has vertex set $I=\lbrace 1, \dots, s \rbrace$ and edge set $E$ such that the number of edges between the $i$-th and the $j$-th vertex is equal to
$$
\begin{cases}
\ext^1(F_i,F_j) & \text{if } i \neq j,\\
\ext^1(F_i,F_i)/2 & \text{if } i=j.
\end{cases}
$$
Consider the quiver $\overline{Q}$ with the same vertex set as $Q(F)$ and edge set $\overline{E}:=E \sqcup E^{\text{op}}$. For $e \in \overline{E}$ denote by $s(e)$ and $t(e)$ the source and the target of $e$, respectively. Then there are $G$-equivariant isomorphisms
$$\text{Rep}(\overline{Q}):=\bigoplus_{e \in \overline{E}}\Hom(V_{s(e)},V_{t(e)}) \cong \Ext^1(F,F), \quad \mathfrak{g}^\vee \cong \Ext^2(F,F)$$
and through them the second order term of the Kuranishi map for $F$
$$\kappa_2 \colon \Ext^1(F,F) \to \Ext^2(F,F)$$
corresponds to the moment map
$$\mu \colon \text{Rep}(\overline{Q}) \to \mathfrak{g} \cong \mathfrak{g}^\vee, \quad (x, y^\vee) \mapsto \mu(x, y^\vee)=\sum_{e \in E} [x_e, y_e^\vee].$$
Note that this result is proved in \cite[Proposition 6.1]{AS} for polystable sheaves on a K3 surface, and the proof adapts line by line to this setting. 

Define the quotient
$$\mathfrak{M}_0:=\mu^{-1}(0) \sslash G$$
with respect to the trivial rational character (see \cite[Section 5]{AS} for more details on quiver varieties). 

The following result is due to \cite{AS, AS_update} as a consequence of formality. 

\begin{prop}[\cite{AS}, Theorem 6.5(i), \cite{AS_update}, Corollary 4.1]
\label{lemma_fromAS}
Let $F$ be a polystable object in $M_\sigma(v)$. Then there is a local analytic isomorphism
$$(\mathfrak{M}_0,0) \cong (M_\sigma(v), [F])$$
of germes of analytic spaces.
\end{prop}
\begin{proof}
By Lemma \ref{lemma_localstructuremoduli} we have a $G$-equivariant local analytic isomorphism
$$(M_{\sigma(v)}(v), [F]) \cong (\kappa^{-1}(0) \sslash G, 0)$$
where $\kappa^{-1}(0) \sslash G=\Spec(\widehat{R}/\mathfrak{a})^{G}$ with $R:=\C[\Ext^1(F,F)]$. Theorem \ref{thm_formality} implies the quadracity property, i.e.\ 
$$\kappa^{-1}(0) \cong \kappa_2^{-1}(0).$$
Since $\kappa_2^{-1}(0) \cong \mu^{-1}(0)$ $G$-equivariantly by \cite[Proposition 6.1]{AS}, we deduce that $M_\sigma(v)$ and $\mathfrak{M}_0$ are isomorphic locally around $[F]$ and $0$, respectively.
\end{proof} 

\subsection{Group actions on categories and Enriques categories}

The definition of Enriques category involves the notions of group action on a category and category of invariant objects. We briefly recall here those definitions, for more details see \cite{BP}.

\begin{dfn}
\label{definition-G-action} 
Let $G$ be a finite group. Let $x$ be an object of an $\infty$-category $\AAA$. 
A \emph{$G$-action} on $x$ is a functor $\phi \colon BG \to \AAA$ such that $\phi(*) = x$, where $BG$ denotes the classifying space of $G$ regarded as an $\infty$-groupoid and $* \in BG$ is the unique object. Given a $G$-action $\phi$ on $x$, the  \emph{$G$-invariants} are
$$x^G = \lim(\phi)$$
provided the displayed limit exists. 
\end{dfn}

In the next we only consider group actions on $\C$-linear categories, i.e.\ small idempotent-complete stable $\infty$-categories equipped with a module structure over the $\infty$-category of finite complexes of finite dimensional $\C$-vector spaces. By \cite{Cohn} the theory of $\C$-linear categories is equivalent to that of small pretriangulated DG categories over $\C$. In this setting, the invariants (called the invariant category) exist since the $\infty$-category of $\C$-linear categories admits all limits by \cite{Mathew}.   

A $G$-action on a $\C$-linear category $\DD$ induces a $G$-action on its homotopy category $h\DD$, which is equivalent to the data of a group morphism $G \to \Aut(h\DD)$ (see \cite[Section 3.1.1]{BP}). Conversely, under some suitable conditions, a $G$-action on a $\C$-linear triangulated category $\DD$ can be lifted to an enhancement of $\DD$ (see \cite[Corollary 3.4]{BP}). We will always work under these assumptions. As a consequence, objects in $\DD^G$ consist of pairs $(F, \theta)$, where $F \in \DD$ and $\theta$ is a collection of isomorphisms compatible with the group structure on $G$ (see \cite[Section 3.1.2]{BP} for the precise definition). The forgetful functor is given by
$$\Forg \colon \DD^G \to \DD, \quad (F, \theta) \mapsto F.$$

In the following definition, we consider an enhanced triangulated category and by a $G$-action on a triangulated category we mean an action that lifts to an $\infty$-action on the enhancement. 

\begin{dfn}
\label{definition-CY2-Enriques}
Let $X$ be a smooth projective variety over $\C$. An admissible subcategory $\CC$ of $\Db(X)$ is an \emph{Enriques category} if it is equipped with a $\Z / 2\Z$-action whose generator $\tau$ is a nontrivial autoequivalence of $\CC$ satisfying $\S_{\CC} \simeq \tau \circ [2]$, where $\S_{\CC}$ is the Serre functor of $\CC$.
\end{dfn}

Recall that a $\C$-linear category $\DD$ is a CY2 category if its Serre functor satisfies $\S_{\DD} \simeq [2]$. Generalizing the geometric case of Enriques surfaces and their associated covering K3 surface, we have the following result.

\begin{lem}[{\cite[Lemmas 4.5 and 4.6]{BP}}]
\label{lemma-Enriques-CY2}
Let $\CC$ be an Enriques category and $\DD = \CC^{\Z/2\Z}$ be its invariant category for the $\Z/2\Z$-action on $\CC$. Then $\DD$ is a CY2 category, called the \emph{CY2 cover} of $\CC$. Moreover, $\DD$ is equipped with a natural $\Z/2\Z$-action, called the \emph{residual $\Z/2\Z$-action}, such that there is an equivalence $\CC \simeq \DD^{\Z/2\Z}$. 
\end{lem}

\subsection{Proof of Theorem \ref{thm_formality_Enriques_intro}}

Let $\CC$ be an Enriques category. In all examples that we will consider, its CY2 cover $\DD$ is an admissible subcategory of the bounded derived category of a smooth projective variety. Through the equivalence $\CC \simeq \DD^{\Z/2\Z}$ we can identify every object in $\CC$ with an object in the invariant category $\DD^{\Z/2\Z}$. 

On the level of DG categories, the theory of equivariant categories has been developed by \cite{Elagin}. Starting with a $\Z/2\Z$-equivariant DG enhancement $\DD^{dg}$ of $\DD$ (i.e.\ a DG enhancement of $\DD$ with a $\Z/2\Z$-action), he constructed a specific DG enhancement $(\DD^{\Z/2\Z})^{dg}$ of $\DD^{\Z/2\Z}$, given by taking the equivariant category of $\DD^{dg}$ and passing to perfect objects. For $(F, \theta) \in \DD^{\Z/2\Z}$, the sets $\Hom_{(\DD^{\Z/2\Z})^{\text{dg}}}((F, \theta), (F, \theta))$ and $\Hom_{\DD^{\text{dg}}}(F, F)$ are DG algebras, and we can define the usual Lie bracket on them by
\begin{equation} \label{eq_bracket}
[f,g]:=f \circ g - (-1)^{|f||g|}g \circ f.
\end{equation}
Here $\circ$ is the composition on $\Hom_{(\DD^{\Z/2\Z})^{\text{dg}}}((F, \theta), (F, \theta))$ or $\Hom_{\DD^{\text{dg}}}(F, F)$, respectively. It follows from the definition and the fact that $\circ$ is a morphism of complexes that $[-,-]$ is a Lie bracket.

In the next statement, we compare the formality of $\Hom_{(\DD^{\Z/2\Z})^{\text{dg}}}((F, \theta), (F, \theta))$ and $\Hom_{\DD^{\text{dg}}}(F, F)$. By abuse of terminology, we will refer to the formality of $\Hom_{\CC}((F, \theta), (F,\theta))$ and the formality of $\Hom_{\DD}(F,F) \cong\RHom(F,F)$, respectively.

\begin{thm} \label{thm_formality_Enriques}
Let $\CC$ be an Enriques category with CY2 cover category $\DD$. Let $\DD^{\emph{dg}}$ be an equivariant DG enhancement of $\DD$ and $(\DD^{\emph{dg}})^{\Z/2\Z}$ the associated enhancement of $\DD^{\Z/2\Z}$. Let $(F, \theta)$ be an object in $\DD^{\Z/2\Z}\simeq\CC$ such that $\Hom_{\DD}(F,F)$ is formal. Then $\Hom_{\CC}((F, \theta), (F,\theta))$ is formal.   
\end{thm}
\begin{proof}
For every $(F, \theta) \in \CC \simeq \DD^{\Z/2\Z}$, by the construction of \cite[Definition 8.8]{Elagin}, the forgetful functor induces a morphism of DG algebras 
$$f \colon \Hom_{(\DD^{\text{dg}})^{\Z/2\Z}}((F, \theta), (F, \theta)) \to \Hom_{\DD^{\text{dg}}}(F,F).$$
Since $\Forg$ is faithful, we have that $f$ is injective. Moreover, since $\Forg$ commutes with the composition $\circ$, it induces a morphism of DG-Lie algebras by \eqref{eq_bracket}. By the proof of \cite[Lemma 5.8]{PPZ} (which is done on the level of triangulated categories, but the argument works line by line at the DG level) there is a splitting
$$\Hom_{\DD}(F,F) \cong \Hom_{\CC}((F, \theta), (F, \theta)) \oplus \Hom_{\CC}((F, \theta), (F, \theta \otimes \chi)) \, \footnote{We denote by $\chi$ the nontrivial character of $\Z/2\Z$ and $\theta \otimes \chi$ is the linearization given by multiplication with $\chi$.}$$ 
which makes $\Hom_{\CC}((F, \theta), (F, \theta))$ a direct summand of $\Hom_{\DD}(F,F)$ as a $\Hom_{\CC}((F, \theta), (F, \theta))$-module. By the formality transfer \cite[Theorem 2.3]{BMM2} (proved originally in \cite[Theorem 3.4]{Manetti2}), we deduce the formality of $\Hom_{\CC}((F, \theta), (F,\theta))$ from that of $\Hom_{\DD}(F,F)$.
\end{proof}

As in the case of K3 surfaces, Theorem \ref{thm_formality_Enriques} allows us to describe the local structure of moduli spaces of semistable objects in Enriques categories.

\begin{prop} \label{prop_localstrmoduli_Enriques}
Under the assumptions of Theorem \ref{thm_formality_Enriques}, let $\sigma$ be a full numerical stability on an Enriques category $\CC$. Let $F$ be a $\sigma$-polystable object in $\CC$. Then there is a local analytic isomorphism
$$(M_\sigma(v), [F]) \cong (\kappa_2^{-1}(0) \sslash \Aut(F), 0)$$
of germes of analytic spaces.  
\end{prop}
\begin{proof}
This follows from Lemma \ref{lemma_localstructuremoduli} and Theorem \ref{thm_formality_Enriques}.    
\end{proof}

We can apply Theorem \ref{thm_formality_Enriques} to interesting Enriques categories arising from geometry.

\begin{cor} \label{cor_qdsandspecGM3}
Let $X$ be a quartic double solid or a Gushel--Mukai threefold or fivefold. Let $\sigma$ be a Serre-invariant stability condition on the Kuznetsov component $\Ku(X)$ of $X$. Then $\RHom(E, E)$ is formal for every $\sigma$-polystable object $E \in \Ku(X)$.   
\end{cor}
\begin{proof}
First note that by \cite{BudurZhang}, if the underlying category has strongly unique DG enhancement, then the formality property is independent of the DG enhancement. This is the case for K3 surfaces \cite{LuntsOrlov} or $\Ku(X)$ of quartic double solids or GM varieties by \cite{LPZ:enhancement}, which gives us the freedom to choose the DG enhancements in the rest of the proof.

The Kuznetsov components of a quartic double solid and of a Gushel--Mukai variety of odd dimension are Enriques categories \cite[Section 4.3]{BP}. Moreover, to prove the statement for Gushel--Mukai variaties of odd dimension, by \cite[Theorem 1.6]{KuzPerry_cones} it is enough to show it in the threefold case.  We denote such Kuznetsov components by $\CC$. By definition, a Serre-invariant stability condition $\sigma_\CC$ on $\CC$ is fixed by the $\Z/2 \Z$-action defined by the Serre functor on $\CC$. Thus it defines a stability condition $\sigma_\DD$ on the associated CY2 cover $\DD$ which is fixed by the residual $\Z/2\Z$-action on $\DD$ by \cite[Theorem 4.8, Lemma 4.11]{PPZ}. Thus if $E$ is $\sigma_{\CC}$-polystable then $\Forg(E)$ is $\sigma_\DD$-polystable, where $\Forg \colon \CC \simeq \DD^{\Z/2\Z} \to \DD$. In particular, $\Forg(E)$ has linearly reductive automorphism group. Then $\Hom_\DD(\Forg(E), \Forg(E))$ is formal by Theorem \ref{thm_formality} in the case of K3 surfaces, and by \cite[Remark 0.1]{AS_update} for special Gushel--Mukai fourfolds.

In order to apply Theorem~\ref{thm_formality_Enriques}, we need to find a $\Z/2\Z$-equivariant DG enhancement $\DD^{dg}$ of $\DD$. For ordinary GM threefolds, the associated $\DD$ is the Kuznetsov component of a special GM fourfold, with $\Z/2\Z$ acting via the geometric involution. Then the standard DG enhancement, given by the restriction of the DG quotient of the DG category of bounded complexes over its full DG subcategory of acyclic ones \cite{Drinfeld}, is $\Z/2\Z$-equivariant. For special GM threefolds or quartic double solids, the $\Z/2\Z$-action on $\Ku(X)$ is via the geometric involution, hence the standard DG enhancement is equivariant, and by \cite[Theorem 4.2]{Elagin} this induces a DG enhancement of $\DD$ which is equivariant under the residual action. And now the statement follows from Theorem \ref{thm_formality_Enriques}.
\end{proof}

\bibliographystyle{alpha}
\bibliography{references}

\newcommand{\etalchar}[1]{$^{#1}$}
\begin{thebibliography}{BLM{\etalchar{+}}21}

\bibitem[AHLH23]{AHLH}
Jarod Alper, Daniel Halpern-Leistner, and Jochen Heinloth.
\newblock Existence of moduli spaces for algebraic stacks.
\newblock {\em Invent. Math.}, 234(3):949--1038, 2023.

\bibitem[AHR20]{AlperHallRydh}
Jarod Alper, Jack Hall, and David Rydh.
\newblock A {L}una \'{e}tale slice theorem for algebraic stacks.
\newblock {\em Ann. of Math. (2)}, 191(3):675--738, 2020.

\bibitem[Alp13]{Alper}
Jarod Alper.
\newblock Good moduli spaces for {A}rtin stacks.
\newblock {\em Ann. Inst. Fourier (Grenoble)}, 63(6):2349--2402, 2013.

\bibitem[AS18]{AS}
Enrico Arbarello and Giulia Sacc\`a.
\newblock Singularities of moduli spaces of sheaves on {K}3 surfaces and {N}akajima quiver varieties.
\newblock {\em Adv. Math.}, 329:649--703, 2018.

\bibitem[AS23]{AS_update}
Enrico Arbarello and Giulia Sacc\`a.
\newblock Singularities of {B}ridgeland moduli spaces for {K}3 categories: an update.
\newblock {\em arXiv:2307.07789}, 2023.

\bibitem[BLM{\etalchar{+}}21]{BLM+}
Arend Bayer, Mart\'{\i} Lahoz, Emanuele Macr\`{i}, Howard Nuer, Alexander Perry, and Paolo Stellari.
\newblock Stability conditions in families.
\newblock {\em Publ. Math. Inst. Hautes \'{E}tudes Sci.}, 133:157--325, 2021.

\bibitem[BMM21]{BMM2}
Ruggero Bandiera, Marco Manetti, and Francesco Meazzini.
\newblock Formality conjecture for minimal surfaces of {K}odaira dimension 0.
\newblock {\em Compos. Math.}, 157(2):215--235, 2021.

\bibitem[BMM22]{BMM1}
Ruggero Bandiera, Marco Manetti, and Francesco Meazzini.
\newblock Deformations of polystable sheaves on surfaces: quadraticity implies formality.
\newblock {\em Mosc. Math. J.}, 22(2):239--263, 2022.

\bibitem[BP23]{BP}
Arend Bayer and Alexander Perry.
\newblock Kuznetsov's {F}ano threefold conjecture via {K}3 categories and enhanced group actions.
\newblock {\em J. Reine Angew. Math.}, 800:107--153, 2023.

\bibitem[BZ19]{BudurZhang}
Nero Budur and Ziyu Zhang.
\newblock Formality conjecture for {K}3 surfaces.
\newblock {\em Compos. Math.}, 155(5):902--911, 2019.

\bibitem[Coh16]{Cohn}
Lee Cohn.
\newblock Differential graded categories are k-linear stable infinity categories.
\newblock {\em arXiv:1308.2587}, 2016.

\bibitem[Dav21]{Davison}
Ben Davison.
\newblock Purity and 2-{C}alabi-{Y}au categories.
\newblock {\em arXiv:2106.07692}, 2021.

\bibitem[Dri04]{Drinfeld}
Vladimir Drinfeld.
\newblock D{G} quotients of {DG} categories.
\newblock {\em J. Algebra}, 272(2):643--691, 2004.

\bibitem[Ela15]{Elagin}
Alexey Elagin.
\newblock On equivariant triangulated categories.
\newblock {\em arXiv:1403.7027}, 2015.

\bibitem[FIM12]{FIM}
Domenico Fiorenza, Donatella Iacono, and Elena Martinengo.
\newblock Differential graded {L}ie algebras controlling infinitesimal deformations of coherent sheaves.
\newblock {\em J. Eur. Math. Soc. (JEMS)}, 14(2):521--540, 2012.

\bibitem[FMM12]{FMM}
Domenico Fiorenza, Marco Manetti, and Elena Martinengo.
\newblock Cosimplicial {DGLA}s in deformation theory.
\newblock {\em Comm. Algebra}, 40(6):2243--2260, 2012.

\bibitem[KL07]{KaledinLehn}
D.~Kaledin and M.~Lehn.
\newblock Local structure of hyperk\"{a}hler singularities in {O}'{G}rady's examples.
\newblock {\em Mosc. Math. J.}, 7(4):653--672, 766--767, 2007.

\bibitem[KP23]{KuzPerry_cones}
Alexander Kuznetsov and Alexander Perry.
\newblock Categorical cones and quadratic homological projective duality.
\newblock {\em Ann. Sci. \'{E}c. Norm. Sup\'{e}r. (4)}, 56(1):1--57, 2023.

\bibitem[Kuz11]{Kuz_basechange}
Alexander Kuznetsov.
\newblock Base change for semiorthogonal decompositions.
\newblock {\em Compos. Math.}, 147(3):852--876, 2011.

\bibitem[Lie06]{Lieblich}
Max Lieblich.
\newblock Moduli of complexes on a proper morphism.
\newblock {\em J. Algebraic Geom.}, 15(1):175--206, 2006.

\bibitem[LO10]{LuntsOrlov}
Valery~A. Lunts and Dmitri~O. Orlov.
\newblock Uniqueness of enhancement for triangulated categories.
\newblock {\em J. Amer. Math. Soc.}, 23(3):853--908, 2010.

\bibitem[LPZ22]{LPZ:elliptic}
Chunyi Li, Laura Pertusi, and Xiaolei Zhao.
\newblock Elliptic quintics on cubic fourfolds, {O}'{G}rady 10, and {L}agrangian fibrations.
\newblock {\em Adv. Math.}, 408:Paper No. 108584, 56, 2022.

\bibitem[LPZ23]{LPZ:enhancement}
Chunyi Li, Laura Pertusi, and Xiaolei Zhao.
\newblock Derived categories of hearts on {K}uznetsov components.
\newblock {\em Journal of the London Mathematical Society}, 108(6):2085--2490, 2023.

\bibitem[LS06]{LehnSorger}
Manfred Lehn and Christoph Sorger.
\newblock La singularit\'{e} de {O}'{G}rady.
\newblock {\em J. Algebraic Geom.}, 15(4):753--770, 2006.

\bibitem[Man99]{Manetti}
Marco Manetti.
\newblock Deformation theory via differential graded {L}ie algebras.
\newblock In {\em Algebraic {G}eometry {S}eminars, 1998--1999 ({I}talian) ({P}isa)}, pages 21--48. Scuola Norm. Sup., Pisa, 1999.

\bibitem[Man15]{Manetti2}
Marco Manetti.
\newblock On some formality criteria for {DG}-{L}ie algebras.
\newblock {\em J. Algebra}, 438:90--118, 2015.

\bibitem[Man16]{Manetti_book}
Marco Manetti.
\newblock Lie methods in deformation theory.
\newblock {\em Not published}, 2016.

\bibitem[Mat16]{Mathew}
Akhil Mathew.
\newblock The {G}alois group of a stable homotopy theory.
\newblock {\em Adv. Math.}, 291:403--541, 2016.

\bibitem[PPZ23]{PPZ}
Alexander Perry, Laura Pertusi, and Xiaolei Zhao.
\newblock Moduli spaces of stable objects in {E}nriques categories.
\newblock {\em arXiv:2305.10702}, 2023.

\bibitem[Zha12]{Zhang}
Ziyu Zhang.
\newblock A note on formality and singularities of moduli spaces.
\newblock {\em Mosc. Math. J.}, 12(4):863--879, 885, 2012.

\end{thebibliography}

\Addresses

\end{document}